\documentclass[preprint,12pt]{elsarticle}

\usepackage{natbib}
\usepackage{graphicx,amsmath,amsfonts,color,amssymb,comment}
\usepackage{tikz}
\usepackage[nameinlink]{cleveref}
\usepackage{kpfonts}
\usepackage{etoolbox}
\usepackage{environ}

\def\bfm#1{\boldsymbol{#1}} 
\def\RR{\mathbb{R}}

\newtheorem{theorem}{Theorem}[section]
\newtheorem{lemma}[theorem]{Lemma}

\newtheorem{corollary}[theorem]{Corollary}
\newproof{proof}{Proof}
\journal{Computer Aided Geometric Design}

\usetikzlibrary{intersections}
\usetikzlibrary{arrows.meta}

\begin{document}

\begin{frontmatter}

\title{Optimal approximation of spherical squares by tensor product quadratic B\'{e}zier patches}

\author[address1,address2]{Ale\v{s} Vavpeti\v{c}}
\ead{ales.vavpetic@fmf.uni-lj.si}

\author[address1,address2]{Emil \v{Z}agar\corref{corauth}}
\ead{emil.zagar@fmf.uni-lj.si}
\cortext[corauth]{Corresponding author}

\address[address1]{Faculty of Mathematics and Physics, University of Ljubljana,
Jadranska 19, Ljubljana, Slovenia}
\address[address2]{Institute of Mathematics, Physics and Mechanics, Jadranska 19, Ljubljana, Slovenia}

\begin{abstract}
  In \cite{Eisele-1994-best-biquadratic}, the author considered
  the problem of the optimal approximation of symmetric surfaces by biquadratic B\'ezier
  patches. Unfortunately, the results therein
  are incorrect, which is shown in this paper by considering the optimal approximation
  of spherical squares. A detailed analysis and a numerical algorithm are given, providing the best
  approximant according to the (simplified) radial error,
  which differs from the one obtained in \cite{Eisele-1994-best-biquadratic}.
  The sphere is then approximated by the continuous spline of two and six tensor product quadratic
  B\'ezier patches. It is further shown that the $G^1$  smooth
  spline of six patches approximating the sphere exists, but
  it is not a good approximation.
  The problem of an approximation of spherical rectangles is
  also addressed and  numerical examples indicate
  that several optimal approximants might exist in some cases,
  making the problem extremely difficult to handle. Finally,
  numerical examples are provided that confirm theoretical results.
\end{abstract}


\end{frontmatter}


\section{Introduction}\label{sec:introduction}

One of the most important issues of computer-aided geometric design (CAGD)
is the approximation of curves and surfaces by simple polynomial-based objects.
It is well known that even fundamental geometric objects such as circular arcs
or spherical patches do not possess exact polynomial representations.
It is thus an interesting and practical issue to find their best polynomial approximants.
Much work has been done on the optimal polynomial approximation
of circular arcs (see, e.g., \cite{Vavpetic-2020-optimal-circular-arcs} and the references
therein). However, much less is known about the optimal approximation of spherical patches.
The main two classes of the polynomial (or spline) surfaces used in CAGD are
triangular patches and tensor product patches.
The optimal approximation of equilateral spherical triangles by triangular
B\'ezier patches was studied recently in \cite{Vavpetic-Zagar-2022-spherical-triangles}.
Approximation of rational tensor-product biquadratic B\'ezier surfaces
(which also includes the sphere) by polynomial tensor product patches
was proposed in \cite{Floater-97-CAGD-Hermite}.
An exact sphere representation by rational S-patches was considered recently in
\cite{Groselj-Sadl-2022-exact-rational-sphere}.
In \cite{Eisele-1994-best-biquadratic},
the author studied the problem of the best approximation
of symmetric surfaces by biquadratic B\'ezier surfaces.
However, we show in this paper that the obtained
results are incorrect since the constructed approximants are not optimal even
in the spherical case. Our primary purpose is to characterize the optimal polynomial approximant
and provide an algorithm for its construction.

The paper is organized as follows. After the introduction in \Cref{sec:introduction}
some basic preliminaries are explained in \Cref{sec:preliminaries}.
The detailed analysis of the construction of the best approximant together with
a numerical algorithm are given in \Cref{sec:main_result}.
The next section considers $G^1$ smooth optimal approximants of the sphere. In \Cref{sec:rectangle} some basic observations
of the approximation of spherical rectangles are provided.
Numerical examples are shown in \Cref{sec:numerical_examples},
and some closing remarks are given in the last section.

\section{Preliminaries}\label{sec:preliminaries}

Let ${\mathcal S}$ be the unit  sphere in $\RR^3$ centered at $\bfm{0}=(0,0,0)$
and let ${\mathcal R}$ be the spherical square for which the projection of its vertices
along the vector $[0,0,1]^T$ are vertices of the square centered at $(0,0)$
with its side equal to $2a$,
where $0<a\leq\tfrac{\sqrt{2}}{2}$. Let ${\mathcal P}$ be a tensor product quadratic B\'ezier patch parameterized as
\begin{equation}\label{eq:p}
    \bfm{p}(u,v)=\sum_{i=0}^2\sum_{j=0}^2 B_i^2(u)B_j^2(v)\bfm{b}_{ij},\quad u,v\in[-1,1],
\end{equation}
where
\begin{equation*}
    B_i^2(u)=\sum_{i=0}^2\binom{2}{i}\left(\frac{1+u}{2}\right)^i\left(\frac{1-u}{2}\right)^{2-i},\quad i=0,1,2,
\end{equation*}
are quadratic Bernstein polynomials parametrized on $[-1,1]$ and $\bfm{b}_{ij}\in\RR^3$, $i,j=0,1,2$, are the corresponding
control points. Our goal is to find the  optimal approximation of ${\mathcal R}$ by ${\mathcal P}$ according
to the radial error
\begin{equation*}
    \max_{(u,v)\in[-1,1]^2}\left|\|\bfm{p}(u,v)\|_2-1\right|
\end{equation*}
or according to the simplified radial error
\begin{equation}\label{eq:simplified_radial}
    \max_{(u,v)\in[-1,1]^2}\left|f(u,v)\right|, \quad f(u,v)=\|\bfm{p}(u,v)\|_2^2-1.
\end{equation}
Note that the latter one is easier to handle since $f$ is a polynomial function of the coefficients of $\bfm{p}$.
Due to this reason, we will consider this error in the following. However, once the problem is solved for the simplified radial error, an identical approach with $f$ replaced by $g=\sqrt{f+1}-1$
can be used to find the best approximant according to the radial error.

In order to get an appropriate approximation of ${\mathcal R}$ it is natural to require that ${\mathcal P}$
interpolates ${\mathcal R}$ at least at its four vertices and that the control points of the boundary curves of ${\mathcal P}$
lie in the planes passing through corresponding two endpoints and the origin $\bfm{0}$. Consequently, by symmetry we have
\begin{align}
    \bfm{b}_{00}&=(-a,-a,\sqrt{1-2a^2}),\quad  \bfm{b}_{20}=(a,-a,\sqrt{1-2a^2}),\quad
    \bfm{b}_{02}=(-a,a,\sqrt{1-2a^2}), \quad  \bfm{b}_{22}=(a,a,\sqrt{1-2a^2}),\nonumber\\
    \bfm{b}_{10}&=\alpha(\bfm{b}_{00}+\bfm{b}_{20}),\quad  \bfm{b}_{01}=\alpha(\bfm{b}_{00}+\bfm{b}_{02}),\quad
    \bfm{b}_{21}=\alpha(\bfm{b}_{20}+\bfm{b}_{22}), \quad  \bfm{b}_{12}=\alpha(\bfm{b}_{02}+\bfm{b}_{22}),
    \label{eq:b_square}
\end{align}
and $\bfm{b}_{11}=\beta(0,0,1)$, for some unknown real parameters $\alpha,\beta$.
It is quite clear that $\alpha,\beta>0$.
Thus $\bfm{p}=\bfm{p}(\cdot,\cdot,\alpha,\beta)$ and consequently
$f=f(\cdot,\cdot,\alpha,\beta)$ which implies that
we are looking for a minimum of \eqref{eq:simplified_radial}, i.e.,
\begin{equation*}
    \min_{\alpha,\beta>0}\max_{(u,v)\in[-1,1]^2}\left|f(u,v,\alpha,\beta)\right|.
\end{equation*}
Let $\angle=\{(u,v)\in[-1,1]^2;u=v\text{ or }v=-1\}$ and define
\begin{equation*}
  f_s(u,\alpha)=f(u,-1,\alpha,\beta),\quad
  f_d(u,\alpha,\beta)=f(u,u,\alpha,\beta), \quad
  {\rm and\ } f_{s,d}=f|_{\angle}.
\end{equation*}
We shall first find parameters $\alpha^*$ and $\beta^*$ which minimize the maximum of
$|f|$ over
$\angle$ and finally show that they provide global minimization of $|f|$ over $[-1,1]^2$, i.e.,
\begin{equation*}
    (\alpha^*,\beta^*)={\rm{argmin}}_{\alpha,\beta>0}
    \max_{(u,v)\in[-1,1]^2}|f(u,v,\alpha,\beta)|.
\end{equation*}
Note that $f_s(\cdot,\alpha)$ and
$f_d(\cdot,\alpha,\beta)$ are even functions on $[-1,1]$ and it is thus enough to analyse them only on $[0,1]$ if it is more convenient.

The author in \cite{Eisele-1994-best-biquadratic}
claims that for the best approximant
$\bfm{p}$ the error
$f_s$ equioscillates. Although this looks reasonable, it is not true
(neither for $f$ nor for $g$) and we shall show
in the following that one can find a better approximant which does not possess this property
(see \Cref{fig:Eisele_vs_ours_side_diagonal}).

\begin{figure}[htb]
\centering
\includegraphics[width=1\textwidth]{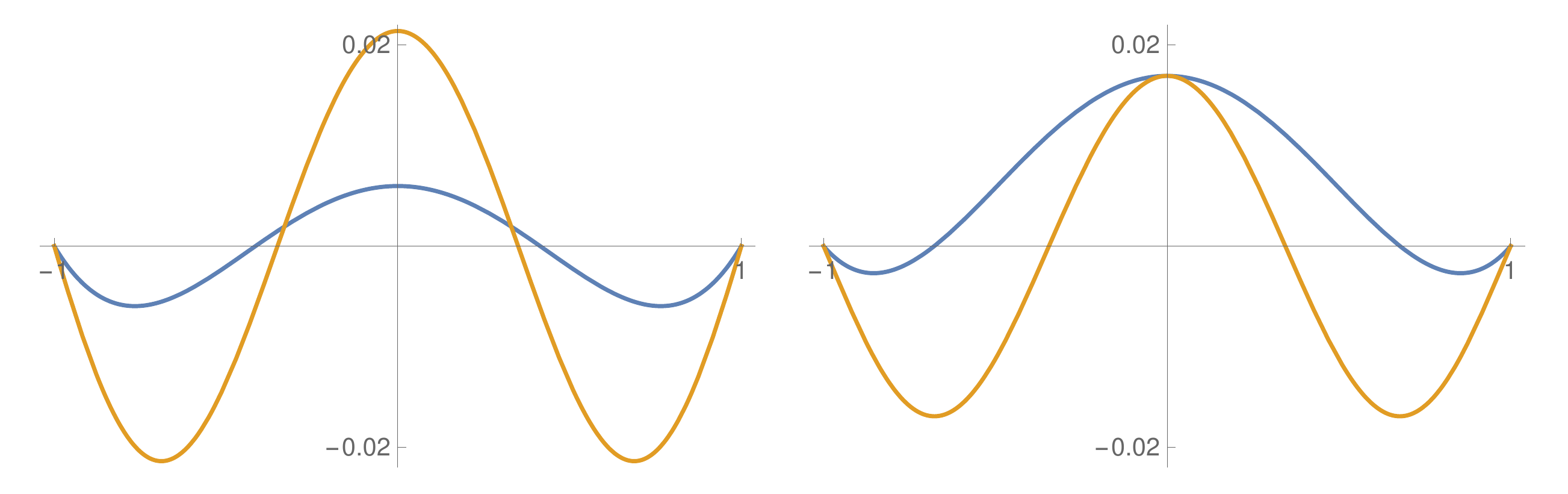}
\caption{Graphs of $f_s$ (blue) and $f_d$ (orange)
for ``the optimal" approximant from \cite{Eisele-1994-best-biquadratic} (left)
and for the optimal approximant constructed in this paper (right).
In all cases $a=1/\sqrt{3}$.
It is clearly seen that
the one on the right has a smaller error on $\angle$.}
\label{fig:Eisele_vs_ours_side_diagonal}
\end{figure}

\section{Main result}\label{sec:main_result}
To prove that the optimal approximant $f_s$ can not equioscillate,
let us start with the following important observation.
\begin{lemma}\label{lem:monotonicity_fs_fd}
  For every $u\in[0,1)$ and $\alpha,\beta>0$ functions $f_s(u,\cdot)$, $f_d(u,\cdot,\beta)$,
  and $f_d(u,\alpha,\cdot)$
  are strictly increasing.
\end{lemma}
\begin{proof}
  The result of the lemma follows directly from the observations
  \begin{align*}
    \frac{\partial f_s}{\partial\alpha}(u,\alpha)
    &=\left(1-a^2\right) \left(1-u^2\right) \left(1+u^2+2 \alpha  \left(1-u^2\right)\right),\\
    \frac{\partial f_d}{\partial\alpha}(u,\alpha,\beta)
    &=\frac{1}{2} \sqrt{1-2 a^2} \beta  \left(1-u^4\right) \left(1-u^2\right)^2+2 \alpha  \left(1-2
   a^2\right) \left(1-u^4\right)^2\\
   &+4 a^2 \alpha  u^2 \left(1-u^2\right)^2+\frac{1}{2} \left(1-2
   a^2\right) \left(1-u^8\right)+u^2 \left(1-u^4\right),\\
   \frac{\partial f_d}{\partial\beta}(u,\alpha,\beta)
    &=\frac{1}{8} \left(1-u^2\right)^2 \left(\beta  \left(1-u^2\right)^2+\sqrt{1-2 a^2}
   \left(1+u^2\right)^2+4 \sqrt{1-2 a^2} \alpha  \left(1-u^4\right)\right).
  \end{align*}
  \qed
\end{proof}
Quite clearly this lemma implies that $f_{s,d}(\cdot,\alpha^*,\beta^*)$ must equioscillate.
Moreover, we will show that  $f_d(\cdot,\alpha^*,\beta^*)$ equioscillates and
$f_s$ does not, which contradict the results in \cite{Eisele-1994-best-biquadratic}.
Let us prove a particular relation between $f_d$ and $f_s$ first.

\begin{lemma} \label{lem:f_s>f_d}
Let $\alpha\in \left[\tfrac 1 2,\tfrac{1}{2-2a^2}\right]$ and let
$f_d(0,\alpha,\beta)\leq f_s(0,\alpha)$. Then
$f_d(u,\alpha,\beta)<f_s(u,\alpha)$ for all $u\in(0,1)$.
\end{lemma}

\begin{proof}
Let us assume first that $f_d(0,\alpha,\beta)=f_s(0,\alpha)$. Then
$$
  \beta=\beta(\alpha)=2 (1+2 \alpha ) \sqrt{1-a^2}-(1+4 \alpha ) \sqrt{1-2 a^2}.
$$
If we define $\alpha(t)=\frac{a^2 }{2 \left(1-a^2\right)}t+\frac{1}{2}$,
then for $t\in[0,1]$ we have $\alpha(t)\in \left[\tfrac 1 2,\tfrac{1}{2-2a^2}\right]$
and it is eqnough to show that $h(t)=f_s(u,\alpha(t))-f_d(u,\alpha(t),\beta(\alpha(t))) > 0$
for all $t\in[0,1]$ and $u\in(0,1)$. But this follows directly from
\begin{align*}
h(t)&=u^2 \left(1-u^2\right) \left(2 \left(1-\sqrt{1-2 a^2} \sqrt{1-a^2}-\frac{3 a^2}{2}\right)
   \left(1-u^2\right) \left(2-u^2\right)+a^2\right) (1-t)\\
   &+\frac{u^2 \left(1-u^2\right)^2}{1-a^2} \left(\left(1-\sqrt{1-2 a^2} \sqrt{1-a^2}-\frac{3
   a^2}{2}\right) \left(4-2 u^2-a^2 u^2\right)+\frac{a^4 u^2}{2}\right) t(1-t)\\
   &+\frac{u^2 \left(1-u^2\right)^2}{4\left(1-a^2\right)^2} \left(\left(1-\sqrt{1-2 a^2}
   \sqrt{1-a^2}-\frac{3 a^2}{2}-\frac{a^4}{8}\right) \left(8-4 a^2\right)
   \left(2-a^2-u^2\right)+\frac{1}{2} a^6 \left(a^2+3 u^2\right)\right) t^2.
\end{align*}
If $f_d(0,\alpha,\beta)<f_s(0,\alpha)=f_d(0,\alpha,\beta(\alpha))$ then
by \Cref{lem:monotonicity_fs_fd} $\beta<\beta(\alpha)$ and we conclude that
$f_d(u,\alpha,\beta)<f_d(u,\alpha,\beta(\alpha))
\leq f_s(u,\alpha)$ for all $u\in(0,1)$.
\qed
\end{proof}
Note that
\begin{equation}\label{eq:fs_on_boundary_of_I}
    f_s\left(u,\frac{1}{2}\right)=-a^2(1-u^2)\leq 0
    \quad
    \text{and}
    \quad
    f_s\left(u,\frac{1}{2(1-a^2)}\right)=\frac{a^4(1-u^2)^2}{4(1-a^2)}\geq 0,
\end{equation}
thus the parameter $\alpha_e$ for which $f_s(\cdot,\alpha)$ equioscillates must be on the interval
$I=\left[\tfrac{1}{2},\tfrac{1}{2(1-a^2)}\right]$. But unfortunately, it may happen that
$\alpha^*\not\in I$. Thus we
define an enlarged interval $I_0=\left[\tfrac{1}{2},\tfrac{3}{2}\right]\supset I$
which will be used in further analysis.

\begin{corollary}\label{col:fs_equioscillates}
If $f_s(\cdot,\alpha_e)$ equioscillates, then for every $\beta>0$
    $$
    \max_{u\in[-1,1]}\left|f_d(u,\alpha_e,\beta)\right|>
    \max_{u\in[-1,1]}\left|f_s(u,\alpha_e)\right|.
    $$
\end{corollary}
\begin{proof}
  Let us suppose that
  $$
    \max_{u\in[-1,1]}\left|f_d(u,\alpha_e,\beta)\right|\leq
    \max_{u\in[-1,1]}\left|f_s(u,\alpha_e)\right|.
    $$
  Since $\max_{u\in[-1,1]}\left|f_s(u,\alpha_e)\right|=f_s(0,\alpha_e)$, we have
  $f_d(0,\alpha_e,\beta)\leq f_s(0,\alpha_e)$ and by \Cref{lem:f_s>f_d}
  $$
    f_d(u,\alpha_e,\beta) < f_s(u,\alpha_e),\quad u\in(0,1).
  $$
  Consequently $\min_{u\in[-1,1]}f_d(u,\alpha_e,\beta)<\min_{u\in[-1,1]} f_s(u,\alpha_e)$
  which is a contradiction.
  \qed
\end{proof}

\begin{lemma}\label{lem:mimmaxfd}
  If $(\alpha^*,\beta^*)$ is an optimal pair of parameters then
  $$
    \min_{u\in[-1,1]} f_{s,d}(u,\alpha^*,\beta^*)
    =\min_{u\in[-1,1]} f_{d}(u,\alpha^*,\beta^*)\quad \text{and} \quad
    \max_{u\in[-1,1]} f_{s,d}(u,\alpha^*,\beta^*)
    =\max_{u\in[-1,1]} f_{d}(u,\alpha^*,\beta^*).
  $$
\end{lemma}
\begin{proof}
  Since minimum and maximum of $f_{s,d}$ is attained either by $f_s$ or $f_d$,
  four possibilities have to be considered.\\
  If
  $$\min_{u\in[-1,1]} f_{s}(u,\alpha^*)
  <\min_{u\in[-1,1]} f_{d}(u,\alpha^*,\beta^*)\quad {\rm and} \quad
  \max_{u\in[-1,1]} f_{d}(u,\alpha^*,\beta^*)
  >\max_{u\in[-1,1]} f_{s}(u,\alpha^*)
  $$
  or
  $$\min_{u\in[-1,1]} f_{d}(u,\alpha^*,\beta^*)
  <\min_{u\in[-1,1]} f_{s}(u,\alpha^*)\quad {\rm and} \quad
  \max_{u\in[-1,1]} f_{s}(u,\alpha^*)
  >\max_{u\in[-1,1]} f_{d}(u,\alpha^*,\beta^*)
  $$
  then since $f_s$ does not depend on $\beta$ and $f_d$ is an increasing function of
  $\beta$, we can find $\beta^{**}<\beta^*$ or $\beta^{**}>\beta^*$ that
  $(\alpha^*,\beta^{**})$ implies a better approximant, respectively.\\
  If $\max_{u\in[-1,1]} f_{s}(u,\alpha^*)>
  \max_{u\in[-1,1]} f_{d}(u,\alpha^*,\beta^*)$
  then
  $$\max_{u\in[-1,1]}f_s(u,\alpha^*)=f_s(0,\alpha^*)>
  \max_{u\in[-1,1]} f_{d}(u,\alpha^*,\beta^*)
  $$
  and the case
  $\min_{u\in[-1,1]} f_{s}(u,\alpha^*)
  <\min_{u\in[-1,1]} f_{d}(u,\alpha^*,\beta^*)$
  is not possible due to the \Cref{lem:f_s>f_d}.
  This confirms the result of the lemma.\qed
\end{proof}

The previous lemma reveals that we are looking for an optimal error $f_d$, i.e., an equioscillating $f_d$. This is still a two-dimensional optimization problem.
But we shall prove that for the optimal pair of parameters $\alpha^*$ and
$\beta^*$ we must have $f_s(0,\alpha^*)=f_d(0,\alpha^*,\beta^*)$ which leads to
one-dimensional optimization. Namely, the quadratic function
$\beta \mapsto f_d(0,\alpha,\beta)-f_s(0,\alpha)$ has
positive leading coefficient and negative constant term, therefore
it possesses the unique positive zero
\begin{equation*}
\beta=\beta_0(\alpha)=2 (1+2 \alpha ) \sqrt{1-a^2}-(1+4 \alpha ) \sqrt{1-2 a^2}.
\end{equation*}
This leads to the following observation.
\begin{lemma} \label{lem:atpointzero}
  If $f_s(0,\alpha)=f_d(0,\alpha,\beta)=0$, then $f_s(u,\alpha)\leq 0$ and $f_d(u,\alpha,\beta)\leq 0$
  for all $u\in[-1,1]$.
\end{lemma}
\begin{proof}
 The solution of the system $f_s(0,\alpha)=0$ and
 $f_d(0,\alpha,\beta)=0$ is
 \begin{equation*}
     \alpha=\frac{1}{\sqrt{1-a^2}}-\frac{1}{2}
     \quad {\rm and}\quad
     \beta=\beta_0\left(\frac{1}{\sqrt{1-a^2}}-\frac{1}{2}\right)
     =4+\sqrt{1-2a^2}\left(1-\frac{4}{\sqrt{1-a^2}}\right).
 \end{equation*}
 Consequently
 \begin{equation*}
     f_s(u,\alpha)=-\left(1-\sqrt{1-a^2}\right)^2u^2(1-u^2)\leq 0
 \end{equation*}
 and
 \begin{multline*}
     f_d(u,\alpha,\beta)=
     -\frac{2 u^2 \left(1-u^2\right)}{1-a^2}
     \left(2 \left(1-\sqrt{1-a^2}-\frac{a^2}{2}\right) u^2 \left(u^2
   \left(1-a^2\right)+a^2 \left(1-u^2\right)\right)\right.\\
   +\left.
   \left(1-u^2\right)^2 \left(\sqrt{1-2
   a^2}-\sqrt{1-a^2}\right)^2
  +\left(1-\sqrt{1-2 a^2}-a^2\right) u^2 \left(1-u^2\right)
   \left(1-a^2\right)\right)\leq 0
 \end{multline*}
 for $u\in[-1,1]$.
 \qed
\end{proof}
If the assumption $f_s(0,\alpha)=f_d(0,\alpha,\beta)=0$ is
extended to $f_s(0,\alpha)=f_d(0,\alpha,\beta)\geq 0$, the following result can be confirmed.
\begin{lemma} \label{lem:speedincreasing}
Let $f_d(0,\alpha,\beta)=f_s(0,\alpha)\geq 0$. Then
\begin{equation*}
    \frac{df_d}{d\alpha}(0,\alpha,\beta_0(\alpha))\geq
    \frac{df_d}{d\alpha}(u,\alpha,\beta_0(\alpha))
\end{equation*}
for all $u\in[-1,1]$.
\end{lemma}
\begin{proof}
  Recall that $\alpha\geq \frac 1 2$ and observe that
  \begin{equation*}
      \frac{df_d}{d\alpha}(0,\alpha,\beta_0(\alpha))
      -
    \frac{df_d}{d\alpha}(u,\alpha,\beta_0(\alpha))
    =2 u^2 \left(k(\alpha-\tfrac{1}{2})+n\right),
  \end{equation*}
  where
\begin{align*}
    k=&\left(\sqrt{1-2 a^2}-\sqrt{1-a^2}\right)^2 \left(1-u^2\right) \left(2 \left(1-u^2\right)^2+u^2+u^4\right)\\
    &+ 2 \left(\sqrt{1-2 a^2} \sqrt{1-a^2}\right) u^2 \left(1-u^4\right)+u^6
   \left(1-a^2\right)+a^2 u^2 \left(1-u^2\right)
   \left(3-u^2\right)\geq 0,\\
    n=&2 \left(\sqrt{1-2 a^2}-\sqrt{1-a^2}\right)^2 \left(2-5 u^2+5 u^4-\frac{3 u^6}{2}\right)\\
    &+\sqrt{1-2 a^2} \sqrt{1-a^2} u^2+a^2 \left(u^2-u^4+\frac{u^6}{2}\right)\geq 0.
\end{align*}
\qed
\end{proof}
\begin{corollary}
  If $f_s(0,\alpha)=f_d(0,\alpha,\beta)\ge 0$ then
  $\max f_s(\cdot,\alpha)=\max f_d(\cdot,\alpha,\beta)=f_s(0,\alpha)$.
\end{corollary}
\begin{proof}
Since
$$
f_s\left(u,\frac{1}{\sqrt{1-a^2}}-\frac 1 2\right)=-\left(1-\sqrt{1-a^2}\right)^2 u^2 \left(1-u^2\right)\le 0
$$
the inequality $f_s(0,\alpha)\ge 0$ implies
$\alpha\ge \alpha_0=\frac{1}{\sqrt{1-a^2}}-\frac 1 2$. Since
$\alpha_0\in[\frac{1}{2},\frac{1}{2-2a^2}]$, \Cref{lem:f_s>f_d} implies $f_d(u,\alpha_0,\beta_0(\alpha_0))\le 0$ for all $u\in[-1,1]$. The result now follows from
\Cref{lem:speedincreasing}. \qed
\end{proof}

\begin{lemma} \label{lem:greatergraph1}
Let $f_d(0,\alpha,\beta)>f_s(0,\alpha)\ge 0$. Let $\alpha'$ and $\beta'$ be such that $f_s(0,\alpha')=f_d(0,\alpha,\beta)=f_s(0,\alpha',\beta')$.
Then $f_d(u,\alpha,\beta)\le f_d(u,\alpha',\beta')\le f_d(0,\alpha,\beta)$ for all $u\in[-1,1]$.
\end{lemma}

\begin{proof}
For every $\mu\in [f_s(0,\alpha),f_d(0,\alpha,\beta)]$ let $\alpha(\mu)$ be the only positive solution of $f_s(0,\alpha(\mu))=\mu$, let $\beta(\mu)$ be the only positive solution of $f_d(0,\alpha(\mu),\beta(\mu))=\mu$, and $\beta_\alpha(\mu)$ be the only positive solution of $f_d(0,\alpha,\beta(\mu))=\mu$. Then
\begin{align*}
\alpha(\mu)&=\frac{2 \sqrt{1+\mu } \sqrt{1-a^2}-1+a^2}{2 \left(1-a^2\right)},\\
\beta(\mu)&=\frac{\left(\sqrt{1-2 a^2}+4 \sqrt{1+\mu }\right) \left(1-a^2\right)-4 \sqrt{1+\mu}\sqrt{1-2 a^2} \sqrt{1-a^2}}{1-a^2},\\
\beta_\alpha(\mu)&=4 \sqrt{1+\mu }-(1+4 \alpha ) \sqrt{1-2 a^2}.
\end{align*}
For $\mu_0=f_s(0,\alpha)$, we have $\alpha=\alpha(\mu_0)$, hence $f_d(u,\alpha,\beta_\alpha(\mu_0))= f_d(u,\alpha(\mu_0),\beta(\mu_0))$ for all $u\in[-1,1]$.
So it is enough to prove that $\tfrac{\partial f_d}{\partial \mu}(u,\alpha,\beta(\mu))\le \tfrac{\partial f_d}{\partial \mu}(u,\alpha(\mu),\beta(\mu))$.
Since
$$
\tfrac{\partial f_d}{\partial \mu}(u,\alpha(\mu),\beta(\mu))-\tfrac{\partial f_d}{\partial \mu}(u,\alpha,\beta(\mu))=
\frac{u^2 \left(1-u^2\right) \left(A_1 \left(\frac{3}{2}-\alpha \right)+A_2
   \left(\sqrt{1+\mu }-1\right)+A_3\right)}{\left(1-a^2\right) \sqrt{1+\mu }},
$$
where
\begin{align*}
A_1&=2 \sqrt{1-2 a^2} \left(1-a^2\right) \left(1-u^2\right)^2\ge 0,\\
A_2&=2 \left(1-u^2\right) \left(2 \left(1-u^2\right) \sqrt{1-2 a^2} \sqrt{1-a^2}+2 u^2\left(1-2 a^2\right)+a^2\right)\ge 0,\\
A_3&=2 \sqrt{1-2 a^2} \sqrt{1-a^2} 2 \left(1-u^2\right)^2
   \left(1-\sqrt{1-a^2}\right)\\
   &+2\left(\sqrt{1-a^2} \left(u^4 \left(1-2 a^2\right)+a^2u^2\right)+\left(1-u^2\right) \left(\left(1-2 a^2\right) 2 u^2+a^2\right)\right)\ge 0,
\end{align*}
the lemma follows.
\qed
\end{proof}

It is clear now that parameters $(\alpha^*,\beta^*)$ are such that
$|f|$ restricted on $\angle$ has the smallest possible minimum.
We will now prove that $f(\cdot,\cdot,\alpha^*,\beta^*)$ has local extrema only
on the diagonals and on the sides of the square $[-1,1]^2$.

The function $f(\cdot,\cdot,\alpha,\beta)$ is a symmetric polynomial. Let us reparameterize it by $e_1=u^2+v^2$
and $e_2=u^2v^2$. The map $(u,v)\mapsto (u^2+v^2,u^2v^2)=(e_1,e_2)$ is a bijection between $T=\{(u,v)\in\RR^2; 0<u<1, 0<v<u\}$ and $S=\{(e_1,e_2)\in\RR^2; 0<e_1<2,e_1-1<e_2<\tfrac 1 4 e_1^2,e_2>0\}$.
If we write $h(e_1,e_2,\alpha)=f(u,v,\alpha,\beta_0(\alpha))$,
the map $f(\cdot,\cdot,\alpha,\beta_0(\alpha))$ has local extrema on $T$ if and only if
$h(\cdot,\cdot,\alpha)$ has local extrema on $S$. But $h(\cdot,\cdot,\alpha)$ as a function on
$\RR^2$ has only one local extreme point $(x,y)$, namely
\begin{align*}
x&=\frac{4 (1+2 \alpha ) \left(\alpha  (1-2 \alpha ) \left(2-4 \alpha +a^2 (3+2 \alpha )\right)+\left(1+7 \alpha +8 \alpha ^2-20 \alpha ^3\right) \left(\sqrt{1-2
   a^2}-\sqrt{1-a^2}\right)^2\right)}{(1-2 \alpha )^2 (6 \alpha +1) \left((14
   \alpha +5) a^2+4 \left(\sqrt{1-2 a^2} \sqrt{1-a^2}-1\right) (2 \alpha +1)\right)},\\
y&=\frac{(1+2 \alpha )^2 \left(a^2 \left(1-4 \alpha ^2\right)+2 \left(5+2 \alpha -8 \alpha ^2\right) \left(\sqrt{1-2 a^2}-\sqrt{1-a^2}\right)^2\right)}{(1-2 \alpha )^2 (1+6\alpha ) \left(a^2 (2 \alpha -1)-2 (1+2 \alpha ) \left(\sqrt{1-2 a^2}-\sqrt{1-a^2}\right)^2\right)}.
\end{align*}
We shall see that for $\alpha=\alpha^*$ we have $y>1$,
which implies that $h(\cdot,\cdot,\alpha^*)$ has no local
extrema on $T$.

Let us first prove that $\tfrac{1-\tfrac 1 5a^4}{2(1-a^2)}=:\alpha_l<\alpha^*<\tfrac{1+\tfrac 1 5a^4}{2(1-a^2)}=:\alpha_r$. Since $f_s(0,\alpha_l)=\frac{a^4 \left(5+10 a^2+a^4\right)}{100 \left(1-a^2\right)}\ge 0$, the functions $f_s(\cdot,\cdot,\alpha)$ and $f_d(\cdot,\cdot,\alpha,\beta_0(\alpha))$ have the maximum at $u=0$ for all $\alpha\in [\alpha_l,\alpha_r]$.
We can write
\begin{align*}
f_d\left(\frac 3 4,\alpha_l,\beta_0(\alpha_l)\right)+
f_d\left(0,\alpha_l,\beta_0(\alpha_l)\right)=
-\frac{\eta}{6553600(1-a^2)^2(\zeta\sqrt{1-a^2}\sqrt{1-2a^2}+\xi)}\leq 0,
\end{align*}
where
\begin{align*}
\eta=a^4\sum_{i=0}^8 \eta_i a^{2i}(1-2a^2)^{8-i},\quad
\zeta=\sum_{i=0}^4 \zeta_i a^{2i}(1-2a^2)^{4-i},\quad
\xi=\sum_{i=0}^5 \xi_i a^{2i}(1-2a^2)^{5-i},
\end{align*}
such that $\eta_i, \zeta_i, \xi_i>0$ for all $i$. Hence the minimum of $\left|f_d\left(\cdot,\cdot,\alpha_l,\beta_0(\alpha_l)\right)\right|$ is larger then the maximum, therefore $\alpha^*>\alpha_l$.
On the other hand, we can write
\begin{align*}
f_d\left(u,\alpha_r,\beta_0(\alpha_r)\right)+
f_d\left(0,\alpha_r,\beta_0(\alpha_r)\right)=
\frac{\zeta\sqrt{1-a^2}\sqrt{1-2a^2}+\xi}{100(1-a^2)^2},
\end{align*}
where $\zeta=4 \left(10-5 a^2+a^4\right) u^2 \left(1-u^2\right)^2 \left(5 \left(2-a^2-u^2\right)+a^4 \left(1-u^2\right)\right)\ge 0$.
Since $\sqrt{1-a^2}\sqrt{1-2a^2}\ge 1-\tfrac 3 2 a^2-a^4$ for all $a\in [0,\tfrac{\sqrt 2}2]$, we have
\begin{align*}
f_d\left(u,\alpha_r,\beta_0(\alpha_r)\right)+
f_d\left(0,\alpha_r,\beta_0(\alpha_r)\right)=
\frac{\zeta\sqrt{1-a^2}\sqrt{1-2a^2}+\xi}{100(1-a^2)^2}\ge
\frac{\zeta(1-\tfrac 3 2 a^2-a^4)+\xi}{100(1-a^2)^2}.
\end{align*}
Since we can write
\begin{align*}
\zeta(1-\tfrac 3 2 a^2-a^4)+\xi=
\sum_{i=0}^4\left(\sum_{j=0}^{9-i} k_{i,j}u^{2j}(1-u^2)^{9-i-j} \right)a^{12-2i}(1-2a^2)^i,
\end{align*}
where all coefficients $k_{i,j}>0$, we have $f_d\left(u,\alpha_r,\beta_0(\alpha_r)\right)+
f_d\left(0,\alpha_r,\beta_0(\alpha_r)\right)\ge 0$ for all $u\in[-1,1]$. Hence the maximum of $\left|f_d\left(\cdot,\cdot,\alpha_r,\beta_0(\alpha_r)\right)\right|$ is larger then the minimum, therefore $\alpha^*<\alpha_r$.

Let us write $\alpha=(1-t)\alpha_l+t\alpha_r$ and prove that $y>1$ for all $t\in[0,1]$, in particular, $y>1$ for $\alpha=\alpha^*$. Let us write $y=\tfrac{y_n}{y_d}$. Let us first show that $y_d<0$. We can write
$$
-\frac{5(1-2a^2)}{(1-2\alpha)^2(1+6\alpha)}y_d=\xi+\zeta\sqrt{1-2a^2}\sqrt{1-a^2},
$$
where $\xi=40(1-2a^2)^3+160a^2(1-2a^2)^2+181a^4(1-2a^2)+49a^6+8a^4(1-2a^2)t+2a^6t$ and $\zeta=4(a^4(1-2t)-5(2-a^2)$. Since $\xi>0$ for all $t\in[0,1]$ it is enough to show that $\xi^2-\left(\zeta\sqrt{1-2a^2}\sqrt{1-a^2}\right)^2>0$. This follows from the equality
\begin{multline*}
\xi^2-\left(\zeta\sqrt{1-2a^2}\sqrt{1-a^2}\right)^2=
17 a^{12} (1-2 t)^2+480 a^6(1-a^2) (1-t)+320 t a^6+12 a^{10} t+\\
 25 a^8 (1-2a^2) t+105 a^8 (1-2 a^2) (1-t)+232 a^{10} (1-t)+32 t a^{10} (1-t)>0.
\end{multline*}
Hence to prove $y>1$, we have to show that $y_d-y_n>0$. We can write
\begin{align*}
y_d-y_n&=\xi\sqrt{1-a^2}\sqrt{1-2a^2}-\zeta=
\frac{(\xi\sqrt{1-a^2}\sqrt{1-2a^2}+k)^2-(\zeta+k)^2}{\xi\sqrt{1-a^2}\sqrt{1-2a^2}+\zeta+2k}\\
&\ge \frac{2k\xi\sqrt{1-a^2}\sqrt{1-2a^2}+\xi^2(1-a^2)(1-2a^2)-\zeta^2-2k\zeta}{\xi\sqrt{1-a^2}\sqrt{1-2a^2}+\zeta+2k}\\
&\ge \frac{2k\xi(1-\frac 3 2 a^2-\frac 1 8 a^4-\frac 7 4 a^6)+\xi^2(1-a^2)(1-2a^2)-\zeta^2-2k\zeta}{\xi\sqrt{1-a^2}\sqrt{1-2a^2}+\zeta+2k},
\end{align*}
where $k=\tfrac{103901}{256}$ and
\begin{align*}
\xi&=\frac{1}{625(1-a^2)^4}\left(\sum_{i=0}^{15}\sum_{j=0}^3 \xi_{i,j}a^{2i}(1-2a^2)^{15-i} t^j(1-t)+\sum_{i=0}^{15} \xi_{i,4}a^{2i}(1-2a^2)^{15-i} t^4\right),\\
\zeta&=\frac{1}{625(1-a^2)^4}\left(\frac{1}{256}\left(\sum_{i=0}^{15}\sum_{j=0}^3 \zeta_{i,j}a^{2i}(1-2a^2)^{15-i} t^j(1-t)+\sum_{i=0}^{15} \zeta_{i,4}a^{2i}(1-2a^2)^{15-i} t^4\right)+2k\right),\\
2k\xi&(1-\frac 3 2 a^2-\frac 1 8 a^4-\frac 7 4 a^6)+\xi^2(1-a^2)(1-2a^2)-\zeta^2-2k\zeta\\
&=
\sum_{i=0}^{15}\sum_{j=0}^8 \eta_{i,j}a^{2i+6}(1-2a^2)^{15-i} t^j(1-t)+\sum_{i=0}^{15} \eta_{i,9}a^{2i+6}(1-2a^2)^{15-i} t^9,
\end{align*}
and all coefficients $\xi_{i,j}$, $\zeta_{i,j}$ and $\eta_{i,j}$ are positive. Hence $y>1$ and the maximum and minimum of $f(\cdot,\alpha^*,\beta^*)$ are on the diagonals or on the sides of the square $[-1,1]^2$. Thus we have proved the following crucial theorem.

\begin{theorem}\label{thm:extrema_of_f}
  The maximum and the minimum of $f(\cdot,\cdot,\alpha,\beta)$ are on the set $\angle$.
\end{theorem}

It is quite clear now how to find the best approximant. The procedure is identical for
the simplified radial error and for the radial error by replacing
$f$ by $g$ so we consider only the first one.
\begin{itemize}
    \item[a)] Solve the equation $f_d(0,\alpha,\beta)=f_s(0,\alpha)$ on $\beta$
    to get the admissible $\beta_0(\alpha):=\beta(\alpha)$.
    \item[b)] Solve the system of nonlinear equations
    \begin{align*}
      \frac{\partial f_d}{\partial u}(u,\alpha,\beta_0(\alpha))&=0,\\
      f_d(0,\alpha,\beta_0(\alpha))+f_d(u,\alpha,\beta_0(\alpha))&=0
    \end{align*}
    on $u$ and $\alpha$ to get the admissible solution $u_0$ and $\alpha_0$.
    \item[c)] Return the best approximant $\bfm{p}$ constructed by $\alpha_0$
    and $\beta_0(\alpha_0)$.
\end{itemize}
Note that the above algorithm must be performed numerically since the equation
in b) can not be solved in a closed form in general. The Newton-Raphson method might be used
or any other appropriate algorithm for solving the system of nonlinear equations.

\section{Optimal $G^1$ approximation}\label{sec:G1_approximation}

A natural question arises whether the optimal approximant of the spherical square
constructed in the previous section can be used
as a $G^1$ tensor product quadratic B\'ezier spline approximant of the unit sphere.
In this case, either two or six patches
can be put together. We shall focus only on the spline of six patches since even in this case
the obtained $G^1$ surface is not a good approximation of the sphere.\\
If six patches are put together we have $a=\tfrac{\sqrt 3}{3}$. The $G^1$ condition implies
that the tangent plane of each of the three patches meeting at the common corner point
must coincide with the tangent plane of a sphere. Hence
$\tfrac{\partial\bfm{p}}{\partial u}(-1,-1)\times\tfrac{\partial\bfm{p}}{\partial v}(-1,-1)
=[\tfrac 2 3(1-2\alpha)\alpha,\tfrac 2 3(1-2\alpha)\alpha,\tfrac 4 3(1-\alpha)\alpha]^T$ must be parallel to
the unit normal of $\bfm{p}$ at $(-1,-1)$, i.e., parallel to
$[-\tfrac{\sqrt{3}}{3},-\tfrac{\sqrt{3}}{3},\tfrac{\sqrt{3}}{3}]^T$. This implies
$\alpha_{G}=\tfrac 3 4$.
For every $v\in[0,1]$ the tangent plane at points $\bfm{p}(1,v)$ and $R\bfm{p}(0,v)$,
where $R$ is the rotation matrix around $y$-axis for the angle $\tfrac\pi 2$, must coincide.
Due to the symmetry, the normal vector of the tangent plane must lie in the plane defined by $\bfm{b}_{20}$, $\bfm{b}_{22}$, and the origin.
Therefore $\tfrac{\partial\bfm{p}}{\partial u}(-1,v)\times\tfrac{\partial\bfm{p}}{\partial v}(-1,v)$ must be
perpendicular to $\bfm{p}(1,-1)\times \bfm{p}(1,1)$ which implies $\beta_G=\tfrac{7\sqrt{3}}6$.
So the $G^1$ spline approximant is unique and its radial error is
$\left|g(0,0,\alpha_G,\beta_G)\right|=\tfrac{5\sqrt{3}-8}{8}\approx 0.0825$ which is much
bigger than the radial error of the corresponding $G^0$ spline of six optimal approximants
constructed in the previous section (see \Cref{fig:semi_full_sphere} and \Cref{fig:G1_approximant}).

\section{Rectangular case}\label{sec:rectangle}
Another interesting issue is the optimal approximation of a spherical rectangle.
Recall that in the case of the approximation of a spherical square, the minimum
of the simplified radial error was never attained at the boundary of a patch.
The situation is quite different in the case of the approximation of a spherical
rectangle. The explanation will be supported by numerical evidence but no formal proof
will be provided.\\
Let the projection of the vertices of a spherical rectangle along the vector
$[0,0,1]^T$ be vertices of a planar rectangle with its larger edge $2a$ fixed and let
$2b<2a$ be it shorter edge. Similarly, as in the case of a spherical square \eqref{eq:p},
we define a tensor product quadratic
B\'ezier patch $\bfm{p}_r$ as
\begin{equation*}
    \bfm{p}_r(u,v,\alpha_1,\alpha_2,\beta)
    =\sum_{i=0}^2\sum_{j=0}^2 B_i^2(u)B_j^2(v)\bfm{c}_{ij},\quad u,v\in[-1,1],
\end{equation*}
where, analogously as in \eqref{eq:b_square},

\begin{alignat}{3}
    &\bfm{c}_{00}=(-a,-b,\sqrt{1-a^2-b^2}),\quad
    &&\bfm{c}_{10}=\alpha_1(\bfm{c}_{00}+\bfm{c}_{20}),\quad
    &&\bfm{c}_{20}=(a,-b,\sqrt{1-a^2-b^2}),\quad \nonumber\\
    &\bfm{c}_{01}=\alpha_2(\bfm{c}_{00}+\bfm{c}_{02}),\quad
    &&\bfm{c}_{11}=\beta\,(0,0,1),\quad
    &&\bfm{c}_{21}=\alpha_2(\bfm{c}_{20}+\bfm{c}_{22}), \quad
    \label{eq:b_rectangle}\\
    &\bfm{c}_{02}=(-a,b,\sqrt{1-a^2-b^2}),\quad
    &&\bfm{c}_{12}=\alpha_1(\bfm{c}_{02}+\bfm{c}_{22}),\quad
    &&\bfm{c}_{22}=(a,b,\sqrt{1-a^2-b^2}),\nonumber
\end{alignat}
with some $\alpha_1,\alpha_2,\beta>0$. As in \eqref{eq:simplified_radial},
we define $f_r(\cdot,\cdot,\alpha_1,\alpha_2,\beta)
=\left\|\bfm{p}_r(\cdot,\cdot,\alpha_1,\alpha_2,\beta)\right\|^2-1.$
We have seen in the analysis of the square case that the global minimum of $f$ can not appear on the boundary
of the square $[-1,1]^2$. We shall see in the following that this can actually happen in the rectangular case
for some particular ratio of sides $a$ and $b$. Once the global minimum is on the boundary,
numerical examples confirm that several optimal approximants of the spherical rectangle might exist. This
indicates that the rectangular case is a much more challenging problem.\\
Let us now try to find the condition on $b$,
which implies that the global minimum and maximum
of $f$ are on the boundary. If this is the case, we can assume that the restriction of
$f_r$ on $(u,-1)$, $u\in[-1,1]$, equioscillate.
This implies the parameter $\alpha_1$ and the point $u_0\in(-1,0)$
such that $\tfrac{\partial f_r}{\partial u} (u_0,-1,\alpha_1,\alpha_2,\beta)=0$
and $f_r(u_0,-1,\alpha_1,\alpha_2,\beta)$ is the global minimum of $f_r$
for any $\alpha_2$ and $\beta$.
A necessary condition for $f_r(u_0,-1,\alpha_1,\alpha_2,\beta)$
being the global minimum is
$\tfrac{\partial f_r}{\partial v}(u_0,-1,\alpha_1,\alpha_2,\beta)\geq 0$.
Suppose that the triple $(b_a,\alpha_a,\beta_a)$ is the solution of the system of nonlinear
equations
\begin{align*}
&\tfrac{\partial f_r}{\partial v} (u_0,-1,\alpha_1,\alpha_2,\beta)
=0,\\
&f_r(0,-1,\alpha_1,\alpha_2,\beta)
=f_r(0,0,\alpha_1,\alpha_2,\beta),\\
&f(0,-1,\alpha_1,\alpha_2,\beta)
=f(-1,0,\alpha_1,\alpha_2,\beta),\\
\end{align*}
with respect to $b$, $\alpha_2$ and $\beta$.
For every $b=b_a$, the candidate for the best approximant
of the spherical rectangle related to $a$ and $b_a$
is determined by the parameters $(\alpha_1,\alpha_a,\beta_a)$.
If $b<b_a$ all pairs of parameters $(\alpha_2,\beta)$ satisfying the inequalities
\begin{align}
\tfrac{\partial f}{\partial v} (u_0,-1,\alpha_1,\alpha_2,\beta)&\geq 0,\nonumber\\
f(0,-1,\alpha_1,\alpha_2,\beta)&\geq f(0,0,\alpha_1,\alpha_2,\beta),
\label{ineq:rectangular}\\
f(0,-1,\alpha_1,\alpha_2,\beta)&\geq f(-1,0,\alpha_1,\alpha_2,\beta),\nonumber
\end{align}
might induce the best approximant.
Thus, in this case, we might have an uncountable number of optimal
polynomial approximants of the spherical rectangle.
Indeed, let $\bfm{v}_i=(\alpha_{2,i},\beta_i)$,
$i=1,2,3$, be the intersections of two pairs of curves
in \eqref{ineq:rectangular} which are implicitly defined by
taking equalities instead of inequalities. Numerical computations
indicate that all pairs $(\alpha_2,\beta)$ from the triangle
defined by $\bfm{v}_i$, $i=1,2,3$, imply an optimal approximant.
Moreover, numerical computations also indicate that several
optimal approximant might exist if a pair $(a,b)$ is taken from
the grey region in \Cref{fig:ab_domain}.
On \Cref{fig:several_optimal}, radial errors of three optimal approximants of
the spherical rectangle with $a=0.75$ and $b=0.2$ are shown.

\begin{figure}[!htb]
  \centering
  \includegraphics[width=0.45\linewidth]{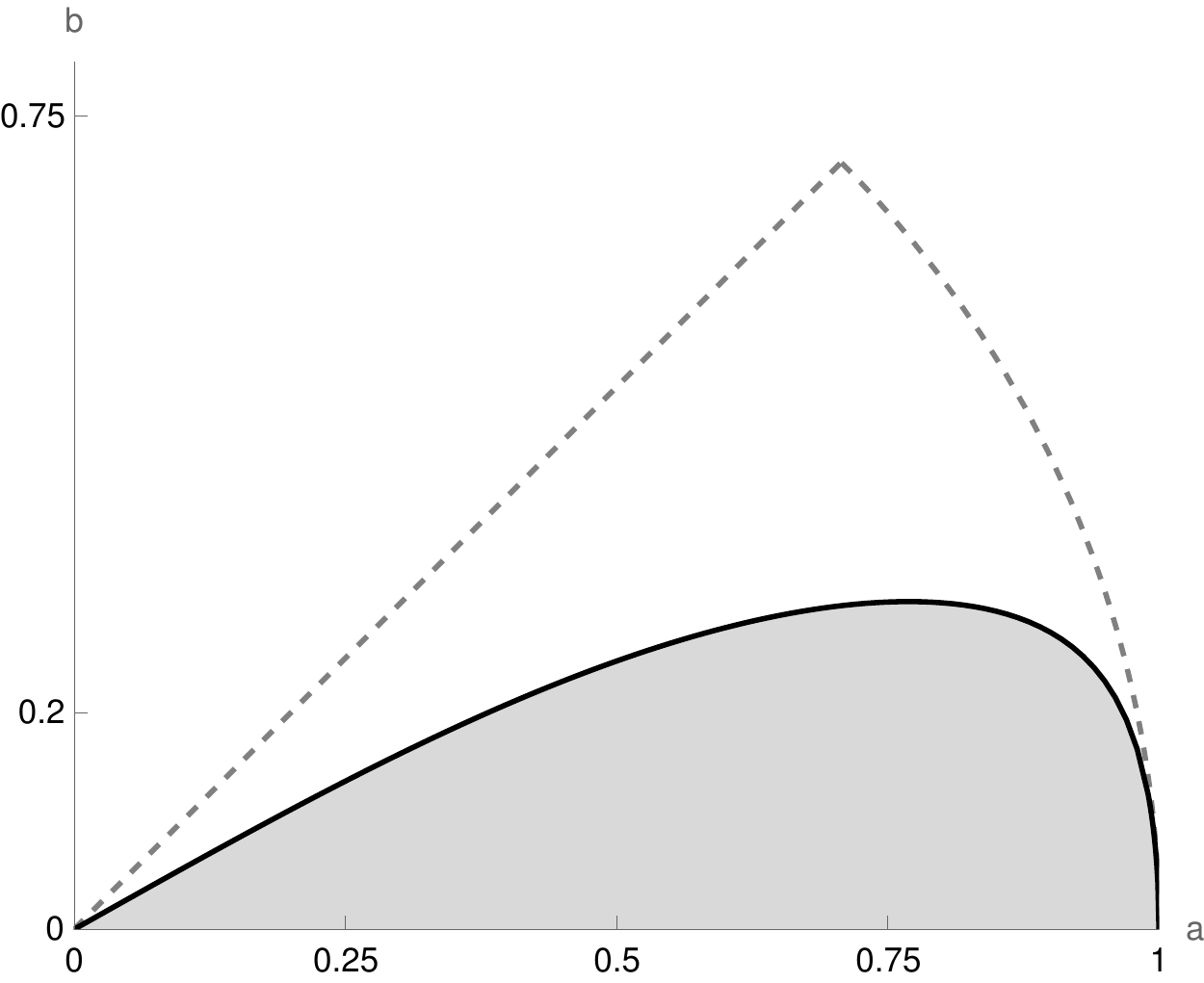}
  \caption{The domain (light grey) in $[0,1]^2$ for which pairs $(a,b)$ might imply several optimal approximants of the spherical rectangle. The grey dashed curve is the graph of the
  function $\min\{a,\sqrt{1-a^2}\}$ and together with the $a$-axis determine the admissible
  domain of pairs $(a,b).$}
  \label{fig:ab_domain}
\end{figure}

\begin{figure}[!htb]
  \centering
  \includegraphics[width=1\linewidth]{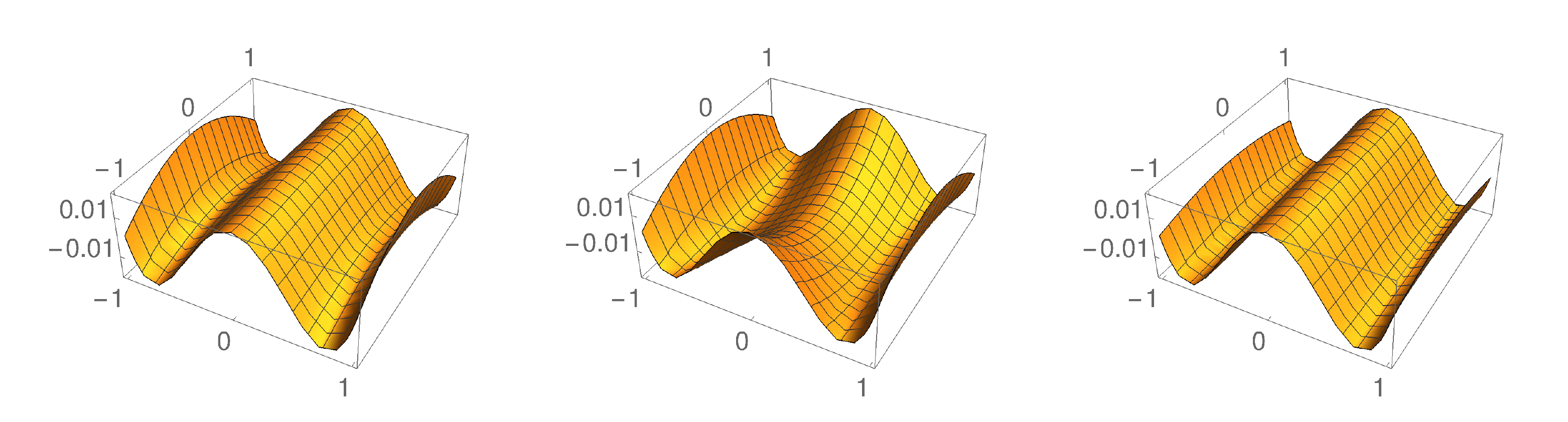}
  \caption{Graphs of $f$ for three optimal
  approximants of the spherical rectangle given by parameters $a=0.75$ and $b=0.2$.
  Here $\alpha_1=1.0277$ and the left graph corresponds to the pair
  $(\alpha_2,\beta)=(0.5313,1.4456)$,
  the middle one to
  $(\alpha_2,\beta)=(0.5313,1.3881)$,
   and the right one to
  $(\alpha_2,\beta)=(0.5239,1.4550)$.}
  \label{fig:several_optimal}
\end{figure}

\section{Numerical examples}\label{sec:numerical_examples}

Some numerical examples will be presented in this section,
confirming proven theoretical results. As the first example,
let us consider the optimal approximation of the
unit  sphere by the quadratic B\'ezier tensor product spline approximant. It is easy to
see that only the spline of two patches, each approximating the hemisphere, and the spline
of six patches, each approximating the one-sixth of the unit sphere, are possible.
Their plots and graphs of radial errors are in \Cref{fig:semi_full_sphere}.
\begin{figure}[!htb]
 \minipage{0.5\textwidth}
    \centering
    \includegraphics[width=0.8\linewidth]{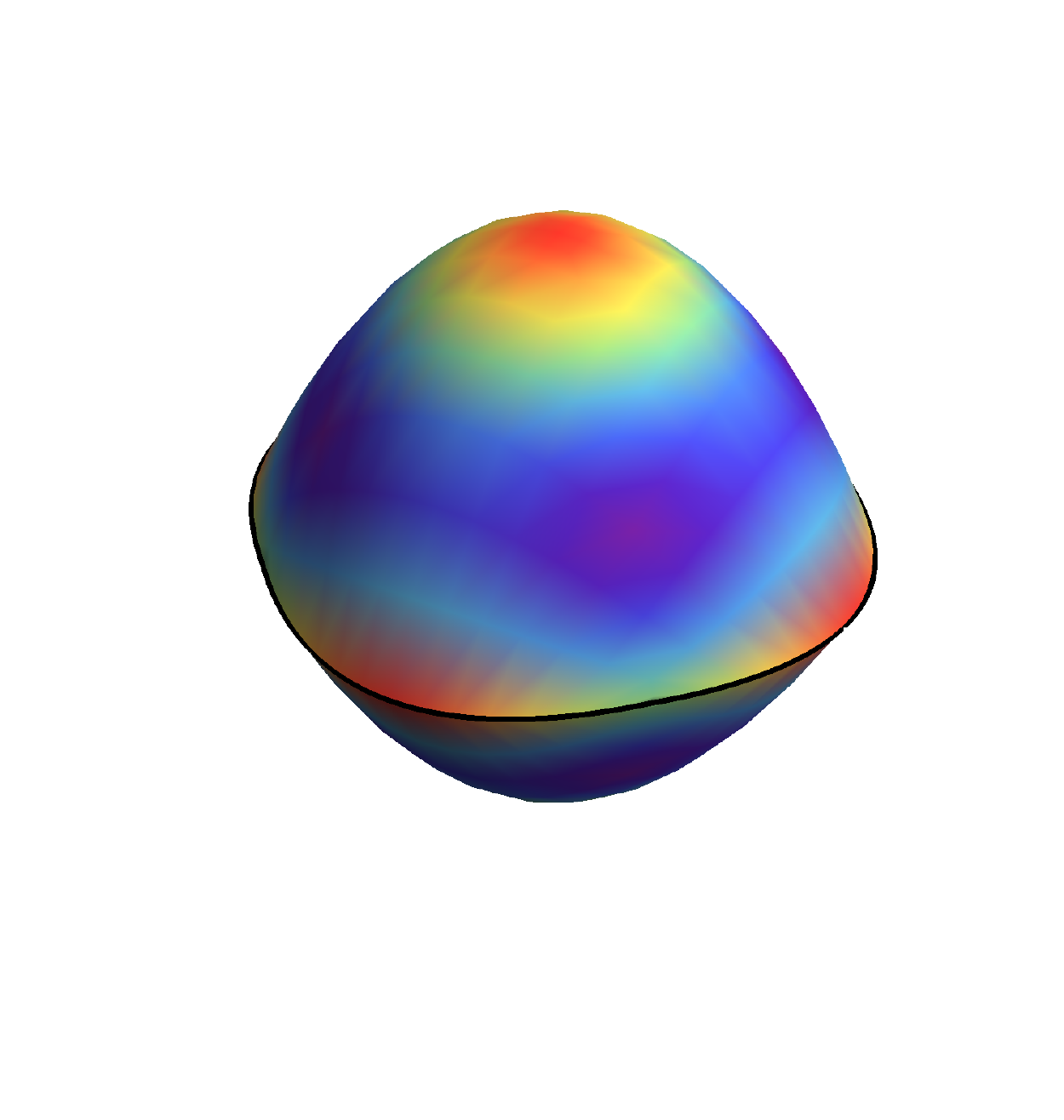}
  \endminipage\hfill
  \minipage{0.5\textwidth}
    \centering
    \includegraphics[width=0.8\linewidth]{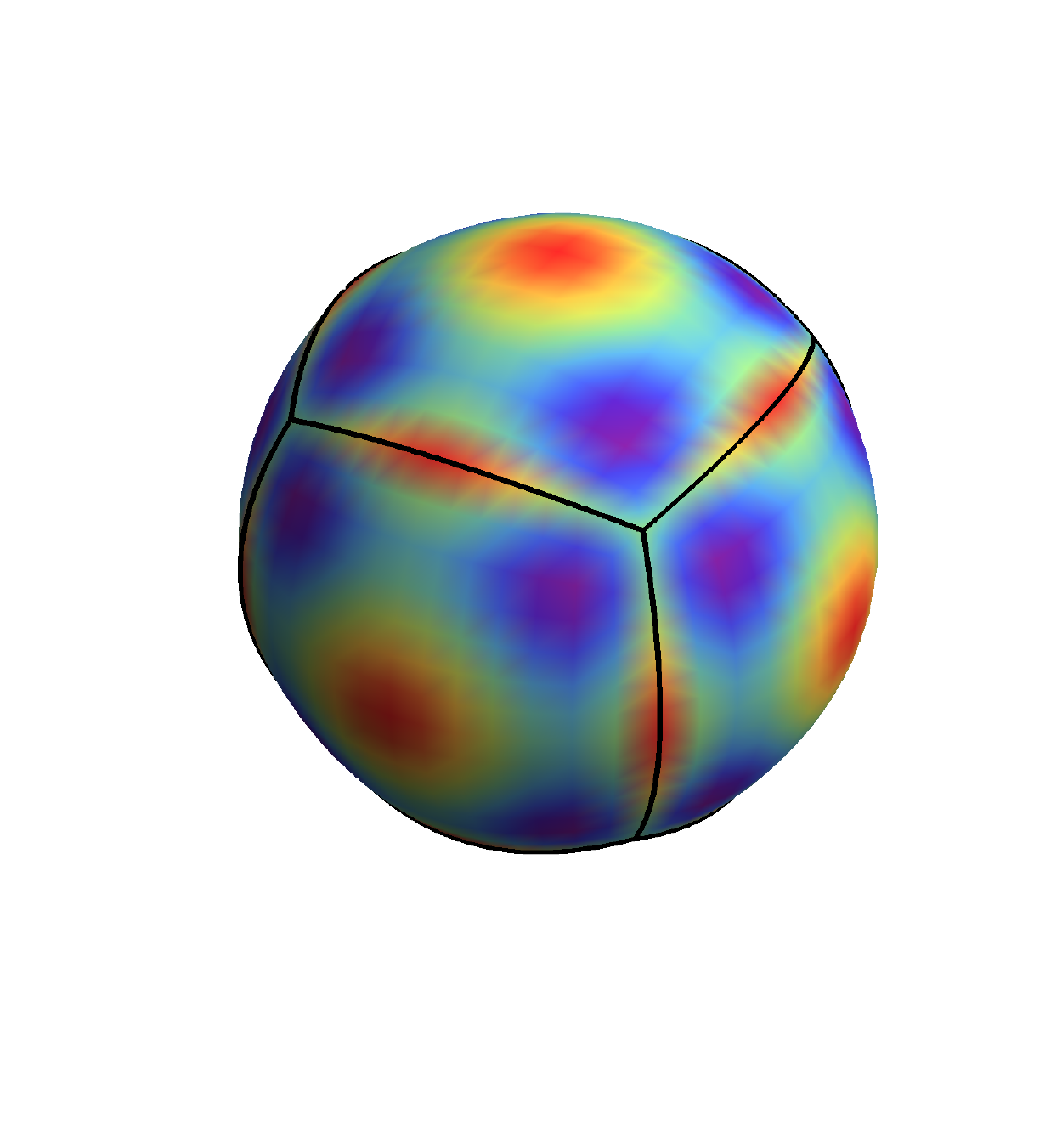}
  \endminipage\hfill
  \minipage{0.5\textwidth}
    \centering
    \includegraphics[width=0.8\linewidth]{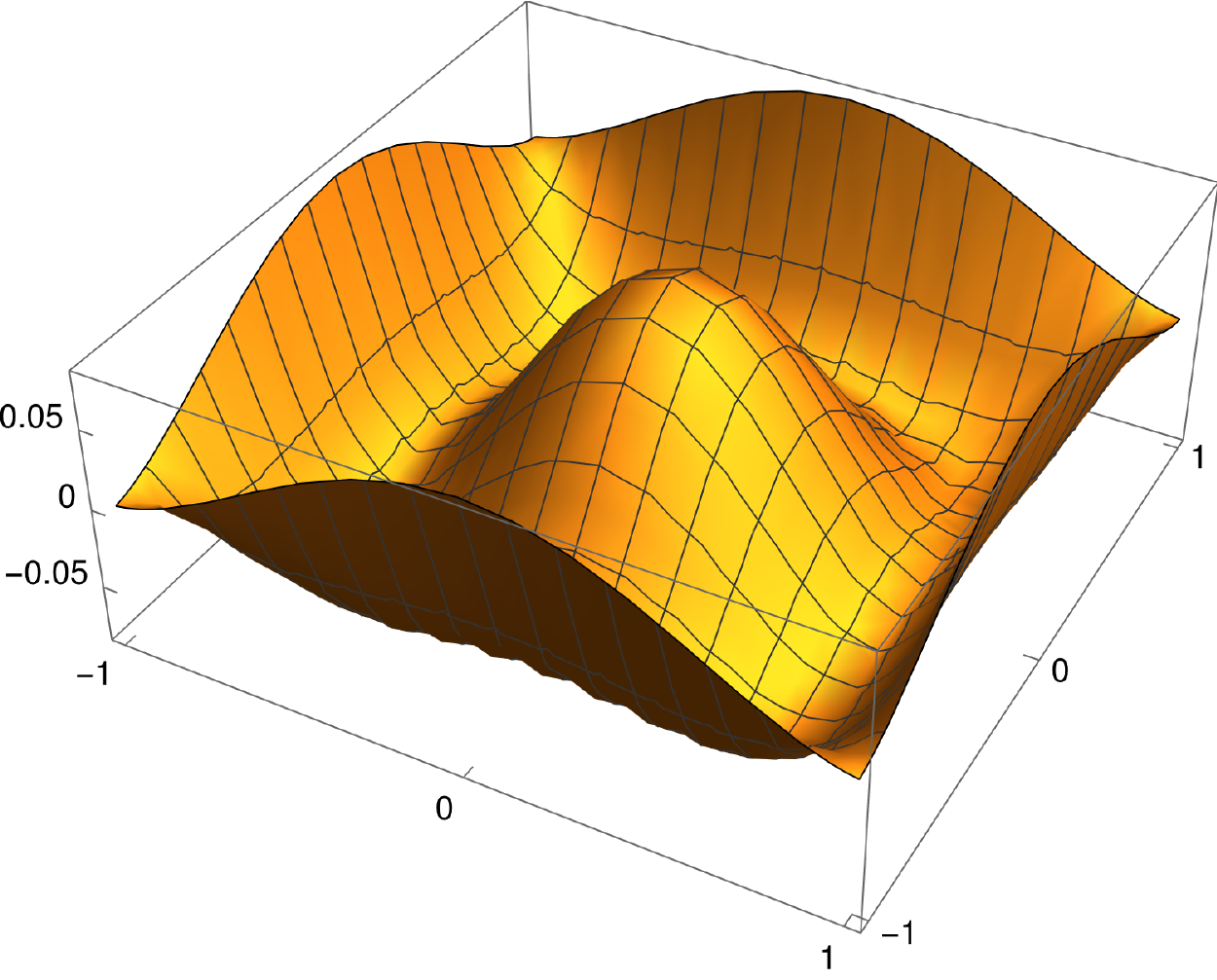}
  \endminipage\hfill
  \minipage{0.5\textwidth}
    \centering
    \includegraphics[width=0.8\linewidth]{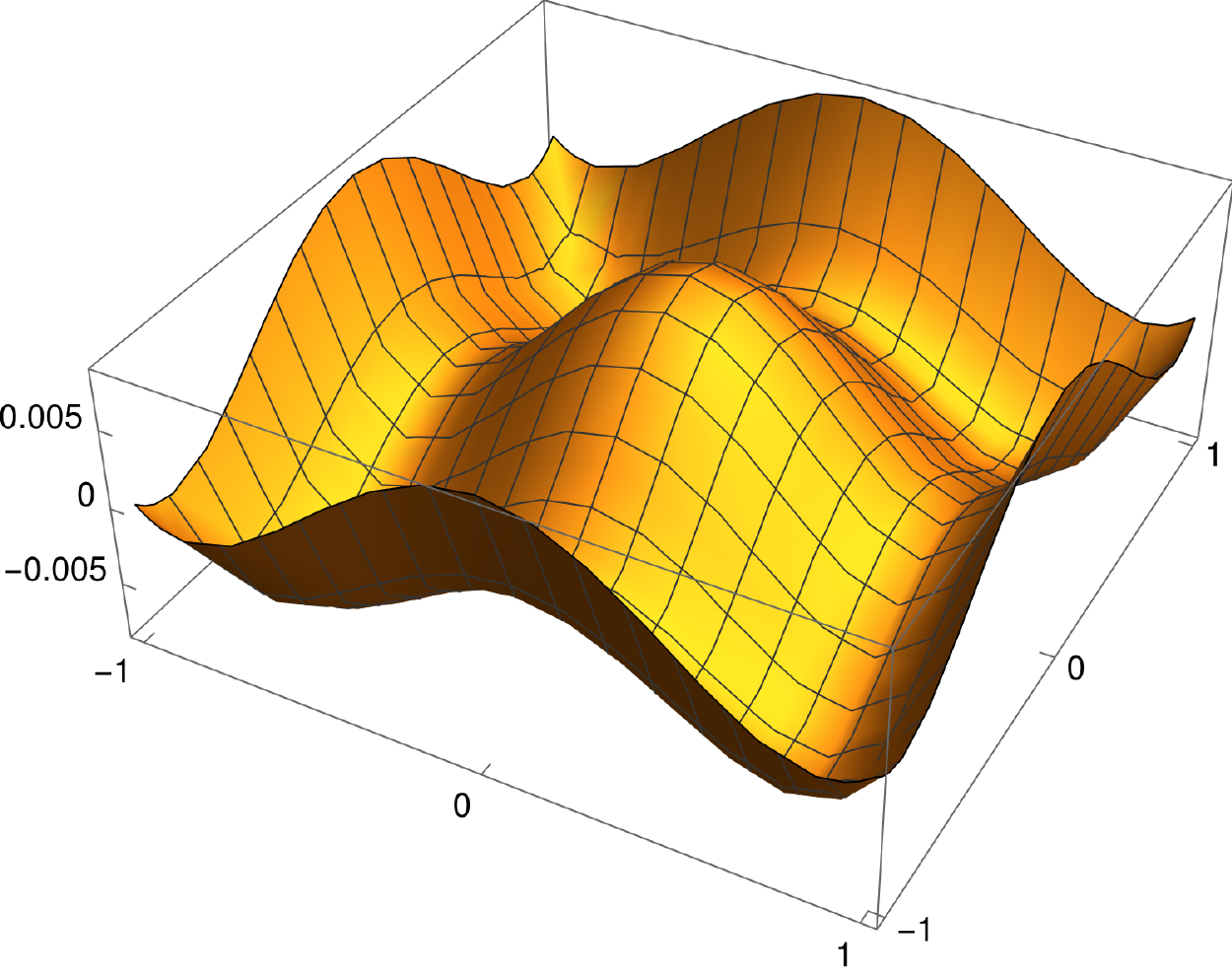}
  \endminipage\hfill
  \caption{Optimal approximation of the unit sphere by the spline of two tensor product
quadratic B\'{e}zier patches (upper left) and by the spline of six
tensor product quadratic B\'{e}zier patches (upper right) together with the graphs of the corresponding radial errors
for the single patch (bottom).
The colours of the approximants indicate the distance
from the sphere (red regions are out of the sphere).
}
  \label{fig:semi_full_sphere}
\end{figure}

In \Cref{tab:rate_of_convergence}, optimal parameters according to the
radial error, radial distances and the numerical convergence rates
are collected for a set of chosen parameters $a$. The numerical rate of convergence is estimated
as follows. Let $d_1$ and $d_2$ be two consecutive radial distances according to the
parameters $a_1$ and $a_2$. Assuming that the radial distance is of the form $d={\rm const}\,a^r$, we can estimate
$r\approx\log\left(d_1/d_2\right)/\log\left(a_1/a_2\right)$.
It is clearly seen that the distance converges to zero as the square of the area of the spherical
square.

\begin{table}[htb]
    \centering
    \begin{tabular}{|r|r|r|r|r|}
         \hline
         \multicolumn{1}{|c|}{$a$}  &
         \multicolumn{1}{|c|}{$(\alpha_r^*,\beta_r^*)$} &
         \multicolumn{1}{|c|}{$|g|$} &
         \multicolumn{1}{|c|}{$r$}\\ \hline
         $a_{max}$    & $(1.0306,4.3393)$  & $8.2331\times 10^{-2}$ & $-$\\ \hline
         $a_{max}/2$  & $(0.5698,1.1630)$  & $6.9966\times 10^{-4}$ & $6.87$\\ \hline
         $a_{max}/4$  & $(0.5160,1.0333)$  & $3.7421\times 10^{-5}$ & $4.23$\\ \hline
         $a_{max}/8$  & $(0.5039,1.0079)$  & $2.2596\times 10^{-6}$ & $4.05$\\ \hline
         $a_{max}/16$ & $(0.5010,1.0020)$  & $1.4005\times 10^{-7}$ & $4.01$\\ \hline
         $a_{max}/32$ & $(0.5002,1.0005)$  & $8.7349\times 10^{-9}$ & $4.00$\\ \hline
         $a_{max}/64$ & $(0.5001,1.0001)$  & $5.4565\times 10^{-10}$ & $4.00$\\ \hline

    \end{tabular}
    \caption{Optimal parameters $\alpha_r^*$ and $\beta_r^*$
    according to the radial error $g$, the corresponding radial distance $|g|$
    and the numerical rate of convergence $r$ for several parameters $a=a_{max}/2^i$,
    $i=0,1,\dots,6$, with $a_{max}=1/\sqrt{2}$.}
    \label{tab:rate_of_convergence}
\end{table}

The approximation of the unit sphere by the $G^1$ spline of six
tensor product quadratic B\'{e}zier patches constructed in \Cref{sec:G1_approximation}
is given in \Cref{fig:G1_approximant}. It is clearly seen that its radial error is much bigger
than the radial error of the optimal approximant in \Cref{fig:semi_full_sphere}. Moreover,
the $G^1$ approximant is one-sided, i.e., the whole approximant is out of the sphere. This suggests that
omitting the interpolation conditions at the vertices of the spherical square would imply better approximant
by pulling it towards the origin.

\begin{figure}[!htb]
 \minipage{0.5\textwidth}
    \centering
    \includegraphics[width=1\linewidth]{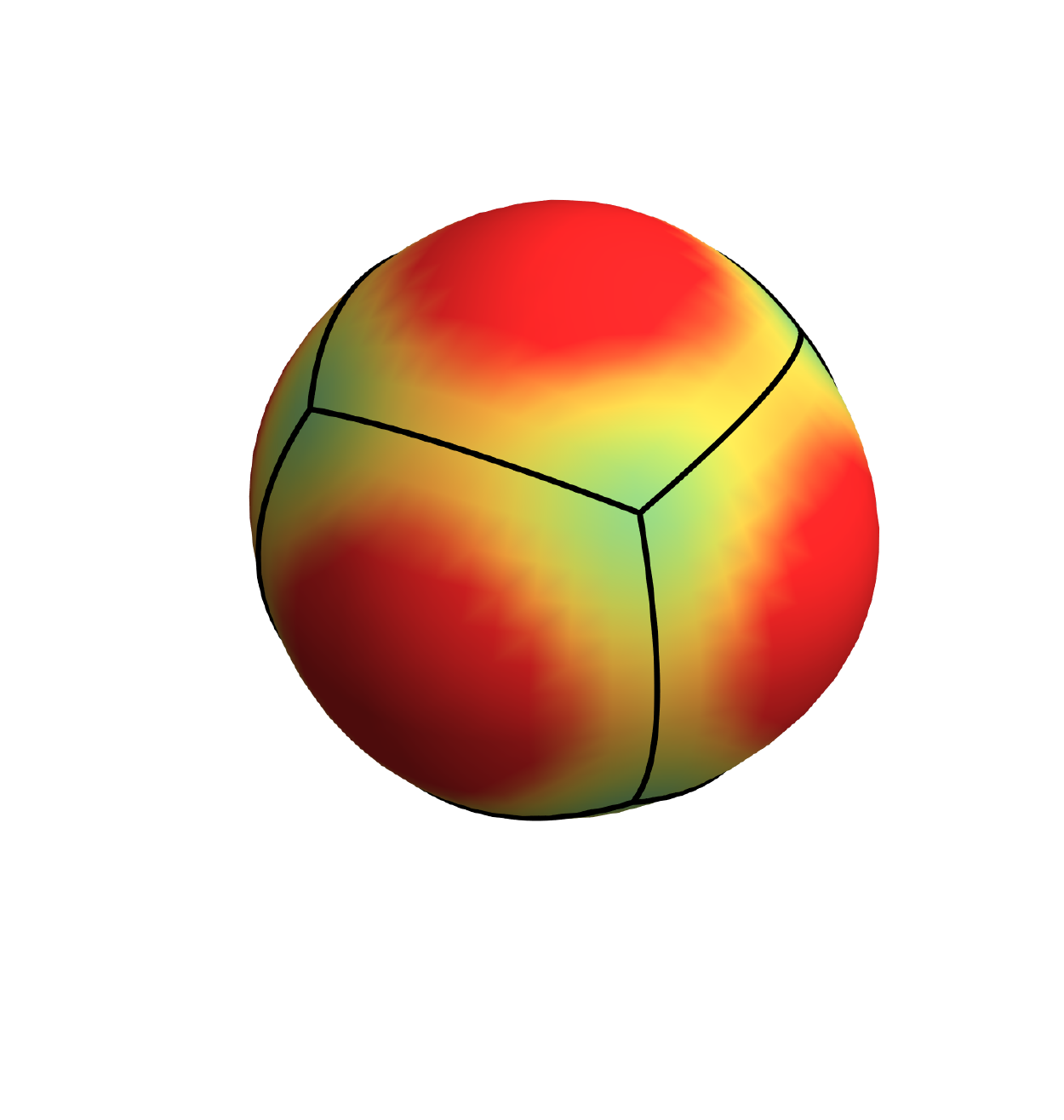}
  \endminipage\hfill
  \minipage{0.5\textwidth}
    \centering
    \includegraphics[width=1\linewidth]{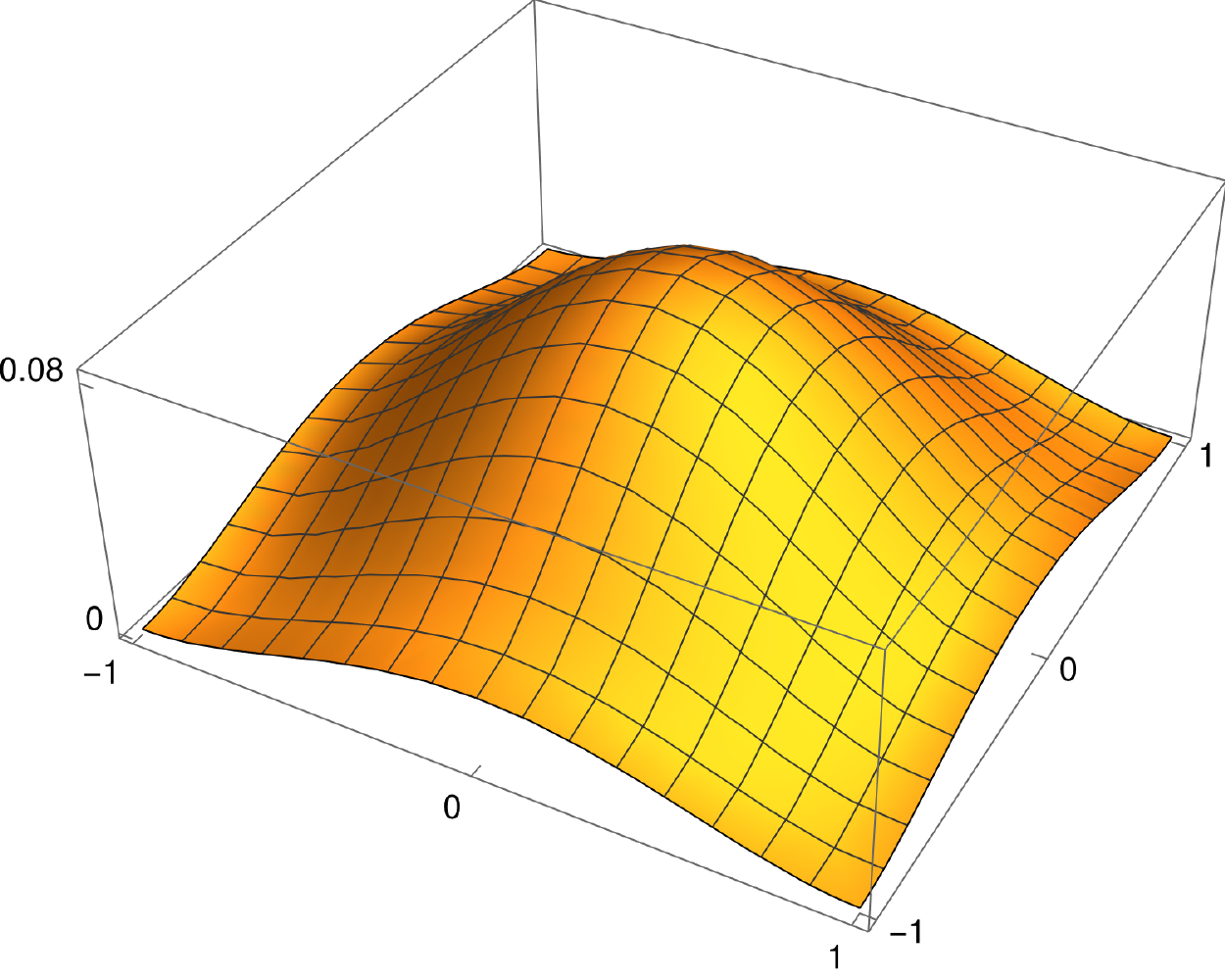}
  \endminipage\hfill
  \caption{Approximation of the unit sphere by the $G^1$ spline of six
tensor product quadratic B\'{e}zier patches (left) together with the graph of the
radial error for the single patch (right).
The colours of the spline approximant indicate its distance from the sphere. The red indicates a bigger distance.}
  \label{fig:G1_approximant}
\end{figure}

\section{Conclusion}\label{sec:conclusion}

Finding the optimal polynomial approximant of a given surface is
a challenging nonlinear optimization problem. There are only a few references
dealing with this problem available. In this paper, we have shown that the results obtained
in \cite{Eisele-1994-best-biquadratic} are not correct. As a counterexample, we found
a better approximant of the spherical square and provided an efficient algorithm for its
construction. It is natural to consider higher degree polynomial approximants of
spherical squares or at least approximants providing smoother polynomial spline patches
($G^k$ continuous tensor product spline patches). Both problems might be considered future work, but they lead to much more complicated nonlinear optimisation issues.
On the other hand, the approximation of spherical rectangles by tensor product quadratic
patches might be of some interest.
Preliminary results reveal that, in some cases,
it leads to several (infinitely many) optimal solutions. Thus the square
and the rectangular case deeply differ in their nature.

{\noindent \sl Acknowledgments.}
The first author was supported by the Slovenian Research Agency program
P1-0292 and the grant J1-4031. The second author acknowledges financial support from the Slovenian Research Agency program grant P1-0288 and the grants N1-0137 and
J1-3005.

\bibliographystyle{elsarticle-harv}

\end{document}